\documentclass[a4paper,11pt]{article}
\usepackage{bbm}
\usepackage{amsfonts}
\usepackage{graphics,color}
\usepackage{amsmath}
\usepackage{amssymb}
\usepackage{mathrsfs}
\usepackage{cite}
\usepackage{verbatim}
\usepackage{float}
\usepackage{graphicx}
\usepackage{amsthm}
\usepackage{textcomp}
\usepackage{subfig}
\usepackage{esint}
\usepackage{enumerate}
\usepackage{hyperref}

\numberwithin{equation}{section}
\mathchardef\emptyset="001F

\newtheorem{Theorem}{Theorem}[section]
\newtheorem{Definition}[Theorem]{Definition}
\newtheorem{Proposition}[Theorem]{Proposition}
\newtheorem{Corollary}[Theorem]{Corollary}
\newtheorem{Lemma}[Theorem]{Lemma}

\newcommand{\nada}[1]{}
\newcommand{\numberset}{\mathbb}

\newcommand{\eps}{\varepsilon}

\newcommand{\N}{\numberset{N}} 
\newcommand{\Om}{\Omega} 
\newcommand{\R}{\numberset{R}}

\newcommand{\rom}{\mathrm}

\DeclareMathOperator*{\esssup}{ess\,sup}

\theoremstyle{definition}

\newtheorem{Example}[Theorem]{Example}

 \title{Energy maximum principles for vectorial higher-order absolute minimisers in $L^\infty$ and $L^p$
}
\author{
Simone Carano, Nikos Katzourakis and Roger Moser}

\begin{document}
\maketitle

\begin{abstract} 
We show that vectorial absolute minimisers of general $k$-th order supremal functionals in $W^{k,\infty}(\Omega,\mathbb R^N)$ satisfy a maximum principle of the form
$$
\max_{\overline U} \rom{H} \big(\cdot, u,  \mathrm D u, ..., \mathrm D^{k}u\big)=\max_{\partial U}\rom{H} \big(\cdot, u,  \mathrm D u, ..., \mathrm D^{k}u\big), \qquad\forall\ U\subseteq\Om \mbox{ open},
$$
suitably interpreted. This is only necessary for absolute minimisers, whilst it characterises a relevant weaker notion of absolute minimality involving compactly supported variations. Further, we obtain an existence result to the Dirichlet problem for such weaker absolute minimisers, as an application of the Baire Category method. Finally, via different methods, we supplement our results by establish a gradient maximum principle for $p$-harmonic maps for $p<\infty$.
\end{abstract}

\section{Introduction}\label{sec:introduction}

For $n,N,k\in\N$, let $\Om\Subset\R^n$ be an open bounded set, and let
\[
\rom{H}:\Om\times\Big(\R^N\times\R^{Nn}
\times\ldots\times\R^{Nn^k}\Big) \longrightarrow \R
\]
be a Carath\'eodory function, where for brevity we symbolise $\R^\Lambda\cong\R^N\times\R^{Nn}\times\ldots\times\R^{Nn^k}$ (namely $\Lambda 
\equiv \Lambda(n,N,k)=N({n^{k+1}-1})({n-1})^{-1}$). Let $\mathscr{L}(\Om)$ be the Lebesgue $\sigma$-algebra of $\Om$ and let also $\mathscr{B}(\R^\Lambda)$ be the Borel $\sigma$-algebra of $\R^\Lambda$. Consider the supremal functional
\begin{align}\label{E infty}
\mathrm E_{\infty}(u, U):=\esssup_{x\in U}\rom{H}\big(x, \mathrm D ^{\vec{k}}u(x)\big),
\end{align}
defined for $u\in W^{k,\infty}(\Om,\R^N)$ and $U\subseteq\Om$ a Lebesgue measurable subset.
In the above, 
\[
\mathrm D^{\vec{k}} u:=(u,\mathrm Du,\ldots, \mathrm D^ku)\ : \ \ \Omega \longrightarrow \R^\Lambda,
\]
denotes the \emph{jet of order $k$} of $u$. One of the main aims of this paper is to establish a maximum principle property for the absolute minimisers of the functional $\mathrm E_\infty$. We recall that $u\in W^{k,\infty}(\Om,\R^N)$ is an \emph{absolute minimiser} of $\mathrm E_\infty$ on $\Omega$ if 
\begin{align}\label{AM def}
\mathrm E_\infty(u,U)\leq \mathrm E_\infty (u+\phi,U),\ \ \forall\ 
 \phi\in W^{k,\infty}_0(U,\R^N),\ \ \forall\  U\subseteq\Om \mbox{ open}.
\end{align}
We observe that this notion is slightly stronger than the standard absolute minimality commonly found in the literature, where minimality is required only on relatively compact open subsets $U \Subset \Omega$. 
\\
Under relatively mild assumptions on the supremand $\rom{H}$, we prove that for an absolute minimiser the essential supremum in \eqref{E infty} is ``attained on the boundary of $U$". In order to give a rigorous meaning to this expression for general $W^{k,\infty}$ mappings, we will utilise the \textit{essential limsup} of $\smash{\rom{H}\big (\cdot,\mathrm D^{\vec{k}}u \big)}$, in order to define the supremum on $\partial U$. More generally, for any $h\in L^\infty(U)$, $U\subseteq \Omega$, we define
\begin{align}\label{rep}
h^*(x):=\lim_{\eps\to   0} \Big(\esssup_{B_\eps(x)\cap U} h\Big),\qquad\forall\  x\in \overline{U}.
\end{align}
The following is the notion of maximum principle that we will utilise herein for (absolute) minimisers of the supremal functional $\mathrm E_\infty$.
\begin{Definition}[Energy maximum principle]
We say that $u\in W^{k,\infty}(\Om,\R^N)$ satisfies the energy maximum principle for $\mathrm E_\infty$ in $\Om$ if
\begin{align*}
\esssup_U \rom{H}\big (\cdot,\mathrm D^{\vec{k}}u \big)=\sup_{\partial U}\rom{H}(\cdot,\mathrm D ^{\vec{k}}u)^*,\qquad\forall\  U\subseteq\Om \text{ open}.
\end{align*}
\end{Definition}
\noindent
Evidently, if $u\in C^k(\overline\Om,\R^N)$ and $\rom{H}\in C(\overline\Om\times\R^\Lambda)$, the above reduces to the more classical pointwise statement
$$
\max_{\overline U} \rom{H}\big (\cdot,\mathrm D^{\vec{k}}u \big)=\max_{\partial U}\rom{H}(\cdot,\mathrm D ^{\vec{k}}u),\qquad\forall\  U\subseteq\Om \mbox{ open}.
$$
Our first main result establishes that this energy maximum principle property is satisfied by absolute minimisers of $\mathrm E_\infty$. More generally, {\it this property holds on any open set $U \subseteq \Omega$ for which $u$ is a minimiser of $\rom E_\infty(\cdot,U)$ with respect to its own boundary data}. However, simple examples show that \emph{it is not sufficient} for the minimality; for instance, the cone function $u(x)=|x|$ belongs to $W^{1,\infty}(\R^n)$ and clearly satisfies the energy (or, equivalently, gradient) maximum principle for $\mathrm E_\infty(u,\Om):=\esssup_\Om|\mathrm Du|$, because $|\mathrm D u|=1$ a.e.\ on $\mathbb R^n$, but $u$ is not a minimiser with respect to its own data on any open $U\subseteq\Om$ which contains the origin (the vertex of the cone). In particular, $u$ is not an absolute minimiser if $\Om$ contains the origin. This raises the question whether we can characterise the class of maps that satisfies the energy maximum principle. We answer this affirmatively, by establishing that {\it the energy maximum principle is a property that characterises a class of maps which is larger than the one of absolute minimisers, and are the ones that minimise $\mathrm E_\infty(\cdot,U)$ with respect to compactly supported variations} in $U$, for every $U\subseteq\Om$.

Before stating our main results precisely, we underline that $\mathrm E_\infty$ in \eqref{E infty} is the general object of study in the $L^\infty$-Calculus of Variations, a field initiated by Aronsson in \cite{A}. His pioneering work on the scalar first order case (namely when $N=k=1$) has been well developed by now, and most challenges have been thoroughly analysed and understood (see \cite{Kbook} for a survey reference). More recently,  the vectorial case ($N>1$) and the higher order case ($k>1$) have begun being explored in a series of works, starting from \cite{K1} and \cite{KM,KP2}. Despite the significant progress in the field, a complete theory is still far from reach. Without any pretension of being exhaustive, we refer also to \cite{K3,K4,KM1} for a glimpse of the literature on vectorial first order problems and to \cite{CKM,DK,KP1,KM2,KM3} for higher order ones. Furthermore, the fractional order case ($k\notin\N$) has been recently explored in \cite{CM}. Let us also note that cornerstone results in the area concerning standard maximum and comparison principles for scalar first order absolute minimisers can be found for instance in \cite{J, ACJS}.  Lastly, further related interesting results in the area appear in the following interesting papers: \cite{AP, A-B, BJ, BDP, BK, CDP, KZ, MWZ, PWZ, PZ, RZ, Ribeiro-Zappale-2024}.

It is worth mentioning that, given boundary data $u_0\in W^{k,\infty}(\Om,\R^N)$, the existence (and/or uniqueness) of absolute minimisers to the Dirichlet problem in $W_{u_0}^{k,\infty}(\Om,\R^N):=u_0+ W^{k,\infty}(\Om,\R^N)$ associated to $\mathrm E_\infty$ in \eqref{E infty} 
 is an open problem already when $k=1$ and $n,N\geq 2$ or when $N=1$ and $n,k\geq 2$. There are some exceptions in the second order case, whenever $\rom H$ has special dependence in the second derivatives of $u$: relevant works are \cite{KM, KM3, KP2, CKM}. Interestingly and perhaps surprisingly, similar results hold true also in the fractional case \cite{CM}.

Now we proceed to specify the assumptions on the function $\rom{H}$. By denoting 
\[
X=(X_0,X_1,\ldots,X_k) \, \in \, \R^N\times \R^{Nn}\times\ldots\R^{Nn^k}\equiv \R^\Lambda, 
\]
we assume that $\rom{H}:\Om\times\R^\Lambda\longrightarrow \R$ is a Carathéodory function (namely measurable with respect to the first variable and continuous with respect to the second), and also \textit{strictly radially increasing} on $\R^\Lambda$, {\it essentially uniformly} with respect to a.e.\ $x\in \Omega$. In symbols, we assume that there exists a continuous function 
\[
c\ :\ [0,1]\times[0,\infty)\longrightarrow [0,\infty), 
\]
satisfying $c(1,\cdot)=c(\cdot,0)=0$ and $c>0$ in $(0,1)\times(0,\infty)$, such that
\begin{align}\label{strong mon}
\rom{H}(x,tX)\leq \rom{H}(x,X)-c(t,|X|),\qquad\forall\  t\in[0,1],\quad\forall\  X\in\R^\Lambda,\quad\mbox{a.e. }x\in \Om.
\end{align}
In other words, $\mathrm H(x,\cdot)$ is a strictly radially increasing function with a modulus of monotonicity independent of $x$.
Note that the radial monotonicity implies that, for almost every $x\in\Om$, $\rom{H}(x,\cdot)$ has a global minimum at $X=0$. Without loss of generality we can assume that $\rom{H}(x,0)=0$ a.e. $x\in\Om$, which implies that $\rom{H}\geq 0$. In this case, we can therefore write
\[
\mathrm E_\infty(u,U)=\|\rom{H}(\cdot, \mathrm D^{\vec{k}}u)\|_{L^\infty(U)},\ \ \ \ \text{for every $U\subseteq\Om$ open.}
\]
The strict radial monotonicity assumption may seem restrictive at first glance, however, the validity of our result extends to a wider class of supremands, even discontinuous ones: indeed, for any function $g:[0,\infty)\longrightarrow [0,\infty)$ which is lower semicontinuous and non-decreasing, set 
\[
\rom G:=g\circ \rom H, \ \ \ \mathrm{\tilde E}_{\infty}(u,U):=\esssup_U \rom G(\cdot,\mathrm D^{\vec{k}} u). 
\]
It is not difficult to see that $\mathrm E_\infty$ and $\mathrm{\tilde E}_{\infty}$ share the same minimisers. In addition, if $u\in W^{k,\infty}(\Om,\R^N)$ satisfies the energy maximum principle for $\mathrm E_\infty$ in $\Om$, then for every $U\subseteq\Om$ open, we have
$$
\mathrm{\tilde E}_{\infty}(u,U)=g\left(\esssup_U \rom{H}(\cdot, \mathrm D^{\vec{k}}u)\right)=g\left(\esssup_{\partial U}\rom{H}(\cdot, \mathrm D^{\vec{k}}u)^*\right)= \esssup_{\partial U}\rom G(\cdot, \mathrm D^{\vec{k}}u)^*,
$$ 
namely $u$ satisfies the energy maximum principle for $\mathrm{\tilde E}_{\infty}$ as well. 

Our second assumption on $\rom{H}$ concerns the essential uniformity of its modulus of continuity, with respect to $x\in\Omega$. More precisely, we assume that for any $R>0$, exists a non-decreasing $\omega_R\in C(0,\infty)$ such that $\omega_R(0^+)=0$ and
\begin{equation}\label{uniformity}
\big|\rom{H}(x,X)-\rom{H}(x,X')\big| \leq \omega_R(|X-X'|),\qquad \text{a.e. } x\in\Omega,\quad\forall\  X\in \overline{\mathbb B}_R(0).
\end{equation}
Since $\rom{H}(x,\cdot)$ is continuous in $\R^\Lambda$, it is uniformly continuous on every closed ball $\overline{\mathbb B}_R(0)\subseteq\R^\Lambda$. Thus, we essentially require that the corresponding modulus of continuity on $\overline{\mathbb B}_R(0)$ is independent of $x\in\Omega$. 

\smallskip

Now we may state our first main result.

\begin{Theorem}[Energy maximum principle for absolute minimisers]\label{thm1}
Let $\Om\Subset\R^n$ be a bounded open set and let $\rom{H}:\Om\times\R^\Lambda\longrightarrow [0,\infty)$ be a Carathéodory function satisfying \eqref{strong mon} and \eqref{uniformity}, with $\rom{H}(x,0)=0$. Assume that $u\in W^{k,\infty}(\Om,\R^N)$ is an absolute minimiser of the functional $\mathrm E_\infty$ (definition \eqref{E infty}). Then $u$ satisfies the energy maximum principle in $\Om$, namely
\begin{align}\label{MP}
\esssup_U \rom{H}\big (\cdot,\mathrm D^{\vec{k}}u \big)=\sup_{\partial U}\rom{H}(\cdot,\mathrm D ^{\vec{k}}u)^*,\qquad\forall\  U\subseteq\Om \mbox{ open}.
\end{align}
Here $\rom{H}(\cdot,\mathrm D ^{\vec{k}}u)^*$ is the essential limsup of $\rom{H}(\cdot,\mathrm D ^{\vec{k}}u)$, defined in \eqref{rep}.
\end{Theorem}
The reader should bear in mind that, more generally, the proof of Theorem \ref{thm1} establishes that the property in \eqref{MP} remains valid for every fixed open set $U\subseteq\Om$ such that $u$ is a minimiser for $\rom E_\infty(\cdot,U)$ among all competitors $v\in W_u^{k,\infty}(U,\R^n):=u+W_0^{k,\infty}(U,\R^n)$. Namely, we have the following more general result.
\begin{Proposition}[Minimality implies energy maximum principle]
Let $U\Subset\R^n$ be a bounded open set and let $\rom E_\infty$ be as in \eqref{E infty}. Let also $\rom{H}:U\times\R^\Lambda\longrightarrow [0,\infty)$ be a Carathéodory function satisfying \eqref{strong mon} and \eqref{uniformity}, with $\rom{H}(x,0)=0$. Assume that $u\in W^{k,\infty}(U,\R^N)$ satisfies
$$
\mathrm E_\infty(u,U)\leq \mathrm E_\infty (u+\phi,U),\ \ \forall\ 
 \phi\in W^{k,\infty}_0(U,\R^N).
$$ 
Then, we have that
$$
\esssup_U \rom{H}\big (\cdot,\mathrm D^{\vec{k}}u \big)=\sup_{\partial U}\rom{H}(\cdot,\mathrm D ^{\vec{k}}u)^*.
$$
\end{Proposition}

In view of the Proposition above, the energy maximum principle can be equivalently interpreted as stating that the \textit{attainment set} $U(u)$ of a minimiser $u$ must contain the boundary of $U$. The attainment set $U(u)$ is the collection of all points $x\in \overline U$ where the essential supremum is ``attained", namely
\[
U(u) := \Big\{ x\in \overline U \ :  \ \rom{H}(x, \mathrm D ^{\vec{k}}u(x))=\esssup_U \rom{H}\big (\cdot,\mathrm D^{\vec{k}}u \big)\Big\},
\]
and it is well defined and non-empty in case $u\in C^k(\overline U,\R^N)$ and $\rom{H}\in C(\overline U\times\R^\Lambda)$. Since we will not exploit this point of view any further in this paper, we refrain from providing any further details. For a definition of $U(u)$ in the non-smooth setting we refer to \cite{BDP}, where the authors show also a minimality property of $U(u)$ for absolute minimisers in the scalar first order case.

As we mentioned earlier, in general the energy maximum principle does not characterise minimality with respect to general $W^{1,\infty}$ variations vanishing on the boundary, but instead it turns out that it characterises a different smaller class of compactly supported $W^{1,\infty}$ variations. Therefore, in order to state our second result, we give the precise definition of this notion. 

\begin{Definition}[The class $\mathrm {AM}_c(\Om,\R^N)$]

We say that $u\in W^{k,\infty}(\Om,\R^N)$ is an absolute minimiser with respect to  compactly supported variations, and we write $u\in \mathrm {AM}_c(\Om,\R^N)$, if
\begin{equation}\label{AM_c}
\mathrm E_\infty(u,U)\leq \mathrm E_\infty (u+\phi,U),\qquad\forall\  \phi\in W^{k,\infty}_c(U,\R^N),\quad\forall\  U\subseteq\Om\mbox{ open}.
\end{equation}
\end{Definition}

We can now complement Theorem \ref{thm1} with the following characterisation.

\begin{Theorem}[Characterisation of $\mathrm {AM}_c(\Om,\R^N)$ via the energy maximum principle]\label{thm2}
Let $\Om\Subset\R^n$ be a bounded open set and let $\rom{H}:\Om\times\R^\Lambda\longrightarrow [0,\infty)$ be a Carathéodory function satisfying \eqref{strong mon} and \eqref{uniformity}, with $\rom{H}(x,0)=0$. Consider the supremal  functional $\mathrm E_\infty$, defined in \eqref{E infty}.
Then, for any map $u\in W^{k,\infty}(\Om,\R^N)$, the following statements are equivalent:
\begin{enumerate}
\item[(i)] $u$ satisfies the energy maximum principle for $\mathrm E_\infty$ in $\Om$, i.e.\ \eqref{MP} holds true;
\item[(ii)] $u$ is an absolute minimiser with respect to compactly supported variations, i.e.\ \eqref{AM_c} holds true.
\end{enumerate}
\end{Theorem}
The proof of Theorems \ref{thm1} and \ref{thm2} are given in Section \ref{sec:proofs}, preceeded by a short preparatory section. In the latter, we recall the notions of strict radial monotonicity and essential limsup, along with some of their fundamental properties. 

Finally, we consider the problem of deriving a corresponding energy maximum principle result for the case of integral $L^p$ functionals for $p<\infty$. In the integral case the situation is quite different conceptually, and therefore for simplicity we restrict our attention to the case of $k=1$ and $\mathrm H$ being the Euclidean norm on the matrix space $\R^{N \times n}$. To this aim, we establish a gradient maximum principle for $p$-harmonic maps. To the best of our  knowledge, a result of this kind is not known in the literature, at least in the vectorial case. In the scalar case, a gradient maximum principle for quasilinear elliptic equations can be found in \cite[Chapter 15]{GT}. We refer also to \cite[Appendix B]{PZZ} regarding gradient maximum and minimum principles for $p$-harmonic potentials on convex rings in the plane. In the case of general $p$-harmonic maps, however, one can only expect a gradient maximum principle to hold, since a minimum principle is not true in general, as existing examples show (see \cite{BBDS}).

Our third and final main result is therefore the following.

\begin{Theorem}\label{thm p}
Let $\Om\subseteq\R^n$ an open set and $p\in[2,\infty)$. Assume that $u:\Om\longrightarrow \R^N$ is a $p$-harmonic map, i.e.\ $u\in W^{1,p}_{\mathrm{loc}}(\Om,\R^N)$ satisfies
\begin{align}\label{p eq}
\mathrm{div} (|\rom Du|^{p-2}\rom Du)=0\qquad\mbox{in }\Om,
\end{align}
in the sense of distributions. Then, for every open set $U\Subset\Om$, we have
\begin{align}\label{GMP}
\max_{\overline U}|\rom Du|=\max_{\partial U}|\rom Du|.
\end{align}
\end{Theorem}

We note that the above maxima are actually well defined since $u\in C^{1,\alpha}_{\mathrm{loc}}(\Om,\R^N)$, for some $\alpha>0$, due to known regularity results for $p$-harmonic maps \cite{U}. The methods we utilise to establish Theorem \ref{thm p} are genuinely different from the ones we utilise to prove the previously stated results for supremal functionals, and are based on the theory of elliptic PDEs.

We note that, even if we restrict Theorem \ref{thm1} to the special case of $k=1$ and of $\mathrm H$ being the Euclidean norm, which reduces to the $\infty$-Dirichlet functional 
\[
\mathrm E_\infty(u,U)=\|\rom D u\|_{L^\infty(U)},  \ \ \ \ U\subseteq \R^n,
\]
the result still has a standalone merit, as it cannot be deduced from Theorem \ref{thm p} by passing to the limit as $p\to \infty$. Indeed, it is not known if the gradient of vector valued absolute minimisers of $\mathrm E_\infty$ can be approximated locally uniformly by the gradient of $p$-harmonic maps (and indeed this cannot be expected to be the case unless $\infty$-harmonic maps are $C^1$, which is known only in the scalar case and in two dimensions \cite{ES}). More generally, the lack of good approximation results for absolute minimisers via $p$-harmonic maps is one of the reasons why existence of absolute minimisers is still open in the vectorial case of $\min\{n,N\}\geq 2$. If however $n=1$ or $N=1$, then special ``scalar" methods can be applied (see e.g.\ \cite{BJW,AK,K6}).

Notwithstanding, unlike the case of general vectorial absolute minimisers, it is not difficult to establish existence results to the higher order Dirichlet problem for absolute minimisers with respect to compactly supported variations. The construction technique is based on the Dacorogna-Marcellini Baire Category method \cite{D-M}. However, there is no hope for a uniqueness result in this class of variations, as this method (which can be seen as an analytic alternative to Gromov's convex integration), always yields infinitely-many solutions. Accordingly, as a consequence of Theorem \ref{thm2}, we have the following result.

\begin{Corollary}[Existence and non-uniqueness in $\mathrm {AM}_c(\Om,\R^N)$]\label{cor}
Let $\Om\Subset\R^n$ be a bounded open set and let $\rom{H}:\Om\times\R^\Lambda\longrightarrow [0,\infty)$ be a continuous function satisfying \eqref{strong mon} and \eqref{uniformity}, with $\rom{H}(x,0)=0$. Assume also that $\rom H(x,X_0,\ldots,X_{k-1}, \cdot)$ is convex in $\R^{Nn^k}$. Consider the functional $\mathrm E_\infty$, defined in \eqref{E infty} and let $u_0\in W^{k,\infty}(\Om,\R^N)$ be fixed. 

Then, there exist infinitely many maps $u\in W^{k,\infty}_{u_0}(\Om,\R^N)\cap \mathrm{AM}_c(\Om,\R^N)$, which are  absolute minimisers with respect to compactly supported variations, satisfying $u=u_0$ on $\partial \Omega$.
\end{Corollary}

The convexity assumption on $\rom H$ in the last variable can be relaxed to Morrey quasiconvexity, provided that the boundary data $u_0$ is piecewise $C^k$ (see \cite{K5}).
However, we underline that the maps $u$ constructed in Corollary \ref{cor} are not absolute minimisers in the sense of Definition \ref{AM def}. In fact, they are not global minimisers, since their energy is strictly greater than the minimal value of $\mathrm E_\infty(\cdot,\Om)$ in $W^{k,\infty}_{u_0}(\Om,\R^N)$. On the other hand, one may wonder if a global minimiser which satisfies the energy maximum principle (or, equivalently, if it is an absolute minimiser with respect to compactly supported variations), is an absolute minimiser. The answer to this question is negative, as the following example shows.

\begin{Example}
Let $n=N=1$ and $\Om=(-1,0)\cup(0,1)$. Consider the minimum problem for $$\mathrm E_\infty(u,U):=\|u'\|_{L^\infty(U)}, \quad U\subseteq\Om,$$ among functions $u\in W^{1,\infty}(U)$ with Dirichlet boundary conditions $$u(-1)=u(0)=0,\quad u(1)=1.$$
The McShane-Whitney extensions $u^{\pm}$ are both global minimisers and they satisfy the energy maximum principle in $\Om$, since $|u^\pm|=1$ a.e. in $\Om$. However, they differ from the absolute minimiser $u$, which clearly satisfies $u=0$ in $(-1,0)$. 
\end{Example}

\section{Preparatory tools}\label{prep}

\subsection{On radial monotonicity}\label{radial mon}

In this first preliminary section, we recall the notions of weak and strict radial monotonicity for a continuous function on $\R^\Lambda$, together with some properties and relevant examples.

\begin{Definition}\label{def rad}
\indent
A continuous function $h:\R^\Lambda\longrightarrow \R$ is said to be weakly (resp. strictly) radially increasing if for every $X\in\R^\Lambda\setminus\{0\}$ the function $t\mapsto h(tX)$ is weakly (resp. strictly) increasing in $[0,\infty)$.
\end{Definition}
Similarly to Example \ref{Ex} below, it can be shown that any level convex function with global minimum at the origin is weakly radially increasing. Of course, a weakly radially increasing function need not be level convex, since its level subsets only need to be star-shaped with respect to the origin.
\\
A continuous function which is strictly radially increasing possesses a modulus of monotonicity, as we can show in the following.

\begin{Proposition}\label{ex c}
Let $h:\R^\Lambda\longrightarrow\R$ be continuous. Then $h$ is strictly radially increasing if and only if there exists a continuous function $c:[0,1]\times[0,\infty)\longrightarrow [0,\infty)$, with $c(1,\cdot)=c(\cdot,0)=0$ and $c>0$ in $(0,1)\times(0,\infty)$, such that
$$
h(tX)\leq h(X)-c(t,|X|)\qquad\forall\  t\in[0,1]\quad\forall X\in\R^\Lambda.
$$ 
\end{Proposition}
\begin{proof}
Since the strict radial monotonicity for $h$ reads as
\begin{align}\label{eqv strict}
t\in(0,1),\, X\in\R^\Lambda\setminus\{0\}\Longrightarrow h(tX)<h(X),
\end{align}
it is evident that the existence of a function $c$ as above implies that $h$ is strictly radially increasing. On the other hand, if $h$ is strictly radially increasing, let us define $c:[0,1]\times[0,\infty)\longrightarrow [0,\infty)$ by
$$
c(t,s):=\inf\{h(X)-h(tX): |X|=s,\, X\in\R^\Lambda\}.
$$
Then, $c$ is continuous, $c(1,s)=0$ for all $s\in[0,\infty)$, and $c(t,0)=0$ for all $t\in[0,1]$. Moreover,   \eqref{eqv strict} and
the compactness of the $s$-spheres in $\R^\Lambda$ imply $c(t,s)>0$ for all $t\in (0,1)$ and $s>0$.
\end{proof}
We note that the existence of a modulus of monotonicity naturally leads to, and justifies, Definition \ref{strong mon}.
\begin{Example}\label{Ex}
Relevant examples of strictly radially increasing functions are:
\begin{itemize}
\item Any strictly level convex function with global minimum at the origin. Indeed, a function $h:\R^\Lambda\longrightarrow \R$ is said to be strictly level convex\footnote{Note that the notion of (strict) level convexity is called (strict) quasi-convexity in \cite{BJW}. } if
$$
h(tX+(1-t)Y)<\max\{h(X),h(Y)\}\qquad\forall X,Y\in\R^\Lambda: X\neq Y,\quad\forall t\in (0,1).
$$
By choosing $Y=0$, we obtain
$$
h(tX)<h(X)\qquad\forall X\in\R^\Lambda\setminus\{0\}\quad\forall t\in (0,1),
$$
which is equivalent to strict radial monotonicity (see \eqref{eqv strict}).
\item Any convex $C^1$-function $h$ with $h(0)=0<h(X)$ for all $X\neq0$. Indeed, by \cite[Thm 2.52]{Dac}, we have $\rom Dh(X)\cdot X\geq h(X)$ for all $X\in \R^\Lambda$, i.e.\ the radial speed of $h$ has a modulus of positivity given by $h$ itself. In particular, $h$ is weakly radially increasing. By integration, for $t\in[0,1]$ and $X\in \R^\Lambda$, we obtain
\begin{align*}
h(X)-h(tX)&=\int_t^{1}\rom Dh(sX)\cdot Xds\\
&\geq\int_t^1\frac1sh(sX)ds\\
&\geq({1-t})h(tX)\\
&\geq({1-t})\min_{|Y|=t|X|}h(Y)=:c(t,|X|).
\end{align*}
The function $c$ is a modulus of monotonicity in the sense of Proposition \ref{ex c}.
\item The cone function $h$ associated to a bounded open set $S\subseteq\R^\Lambda$ star shaped with respect to  the origin. This can be seen in a similar way to the previous example, by simply observing that since $h$ is $1$-homogeneous, one has $\frac{d}{dt}h(tX)=h(X)$ for all $X\in \R^\Lambda$.

\end{itemize}
\end{Example}

\subsection{Properties of the essential limsup}
In this second preliminary section, we give the definition of essential limsup, together with some properties.

\begin{Definition}[Essential limsup]\label{ess limsup}
Let $U\subseteq\R^n$ be a bounded Borel set and $f\in L^\infty(U)$. We say that the function $f^*\in L^\infty(U)$, defined as
$$
f^*(x):=\lim_{\eps\to  0}\esssup_{B_\eps(x)\cap U}f\qquad\forall\  x\in \overline U,
$$
is the essential limsup of $f$.
\end{Definition}
Clearly, if $U$ is open and $f$ is continuous in $U$, then $f^*=f$. 
However, in the more general setting of Definition \ref{ess limsup}, we may have $f\neq f^*$ on a set of positive measure: for instance, consider $f=\mathbbm{1}_{\R^n\setminus K}$ for a nowhere dense compact set $K\subseteq\R^n$, with $\mathscr L^n(K)>0$. In general, from \cite[Proposition 9]{KSIAM}, we have that $f\leq f^*$ almost everywhere and $f^*$ is upper semicontinuous on $U$. Moreover, again by \cite[Proposition 9]{KSIAM}, the definition of $f^*$ allows us to give a pointwise meaning to the essential supremum, indeed for any Borel set $U\subseteq\R^n$ we have
\begin{align}\label{pointwise sup}
\sup_U f^*=\esssup_U f.
\end{align}
For a set $E\subseteq\R^n$ and $\rho>0$, we denote by $E^\rho$ the open $\rho$-neighbourhood of $E$, namely 
\[
E^\rho:=\big\{ x\in\R^n: \mathrm{dist}(x, E)<\rho \big\}.
\]
The following lemma is in order.

\begin{Lemma}\label{lemma esssup}
Let $U\subseteq\R^n$ be a bounded open set. Then, for any $K\subseteq\partial U$ compact subset, we have
$$
\sup_K f^*=\lim_{\rho\to  0} \esssup_{K^\rho\cap U}f.
$$
\end{Lemma}
\begin{proof}
We start with noting that for every $x\in K$ there exists $\rho>0$ such that $B_\rho(x)\cap U\subseteq K^\rho\cap U$, such that
$$
\esssup_{B_\rho(x)\cap U}f\leq \esssup_{K^\rho\cap U}f.
$$
By passing to the limit as $\rho\to  0$ in both sides and taking the supremum over $x\in K$ in the left hand side, we obtain
$$
\sup_K f^*\leq\lim_{\rho\to  0} \esssup_{K^\rho\cap U}f.
$$
It remains to prove the opposite inequality. By definition \ref{ess limsup}, for any fixed $x\in K$ and $\sigma>0$, there exists $\rho_{x,\sigma}>0$ such that 
\begin{align}\label{def of lim}
f^*(x)\geq\esssup_{B_{\rho_{x,\sigma}}(x)\cap U}f-\sigma.
\end{align}
The family of balls $\{B_{\rho_{x,\sigma}}(x); \,x\in K\}$ is an open covering of $K$. Thus, we can extract a finite subcovering $\{B_{\rho_{x_i,\sigma}}(x_i); \,x_i\in K\} _{i=1,\ldots,m}$.
For brevity, let us denote $B_{i,\sigma}=B_{\rho_{x_i,\sigma}}(x_i)$. We claim that there exists $\rho_\sigma>0$ such that $K^{\rho_\sigma}\subseteq \bigcup_{i=1}^m B_{i,\sigma}$. For the sake of contradiction, suppose that for every $j\in\N$ there exists $y_j\in K^{1/j}\setminus\bigcup_{i=1}^m B_{i,\sigma}$. This is equivalent to  
\begin{equation}\label{contr}
\forall\  j\in\N, \quad\exists\, y_j\in\R^n:\left\{
\begin{aligned}
 &y_j\notin B_{i,\sigma}\quad\forall\  i=1,\ldots,m,\\
&\mathrm{dist}(y_j,K)\leq 1/j.
\end{aligned}
\right.
\end{equation}
 Let $\bar y_j\in K$ be a point realising $d(y_j,K)$. Since $K$ is compact, there exists a subsequence $(\bar y_{j_k})\subseteq(\bar y_j)$ such that $\bar y_{j_k}\longrightarrow  \bar x$ as $k\to  \infty$, for some $\bar x\in K$. Moreover, since $\mathrm{dist}(y_{j_k},\bar y_{j_k})\leq1/{j_k}$, we have also that $y_{j_k}\longrightarrow  \bar x$ as $k\to  \infty$. But being $\{B_{i,\sigma}\}_{i=1,\ldots,m}$ a covering of $K$, we must have that eventually $y_{j_k}\in B_{i,\sigma}$ for some $i\in\{1,\ldots,m\}$, which contradicts \eqref{contr}. Now, we define $U_\sigma:=\bigcup_{i=1}^m (B_{i,\sigma}\cap U)$ and note that ${K^{\rho_\sigma}\cap U}\subseteq U_\sigma$. Thus, recalling \eqref{def of lim}, we estimate
\begin{align*}
\lim_{\rho\to  0}\esssup_{K^\rho\cap U}f&\leq\esssup_{K^{\rho_\sigma}\cap U}f\\
&\leq \esssup_{U_\sigma}f\\
&\leq\max_{i=1,\ldots,m}\esssup_{B_{i,\sigma}\cap U}f\\
&\leq\max_{i=1,\ldots,m} f^*(x_i)+\sigma\\
&\leq\sup_K f^*+\sigma.
\end{align*}
By letting $\sigma\to 0$, we conclude the proof.
\end{proof}

\section{Proofs of the main results}\label{sec:proofs}

This section is devoted to the proof of Theorems \ref{thm1} and \ref{thm2}. The proof of the latter is essentially a consequence of the arguments in the proof of the former. Indeed, the key idea in the proof of Theorem \ref{thm1} is to find, for every $U\subseteq\Om$, a suitable competitor in $u_\lambda\in W^{k,\infty}_u(U,\R^N)$ to compare with the absolute minimiser $u$. It turns out that $u_\lambda\in u+W^{k,\infty}_c(U,\R^N)$, and this will be the crucial observation in the proof of Theorem \ref{thm2}.

We start with the proof of Theorem \ref{thm1}.

\medskip

\noindent \textit{Proof of Theorem \ref{thm1}}: Fix $U\subseteq\Om$ an open subset. First, note that for every $x\in\partial U$ and every $\rho>0$, we clearly have
$$
\esssup_{B_\rho(x)\cap U}\rom{H}\big (\cdot,\mathrm D^{\vec{k}}u \big)\leq \esssup_{ U}\rom{H}\big (\cdot,\mathrm D^{\vec{k}}u \big).
$$ 
From Definition \ref{ess limsup}, by passing to the limit as $\rho\to  0$ and then to the supremum over all $x\in\partial U$ in the left hand side, we obtain
\begin{align}\label{one ineq}
\sup_{\partial U}\rom{H} \big (\cdot,\mathrm D^{\vec{k}}u \big)^*\leq\esssup_U \rom{H}\big (\cdot,\mathrm D^{\vec{k}}u \big).
\end{align}
Let us prove the opposite inequality. To this aim we need to use assumptions \eqref{strong mon} and \eqref{uniformity}.\\
We start by setting 
\begin{align}\label{M}
M:=\esssup_U \rom{H}\big (\cdot,\mathrm D^{\vec{k}}u \big).
\end{align}
If $M=0$, there is nothing to prove, so we can assume that $M>0$.
For the sake of contradiction, assume that \eqref{one ineq} holds strictly. Then, by applying Lemma \ref{lemma esssup} to $K=\partial U$, one can find $\delta>0$ such that
$$
\lim_{\rho\to  0}\esssup_{(\partial U)^\rho\cap U}\rom{H}\big (\cdot,\mathrm D^{\vec{k}}u \big)+2\delta\leq M.
$$
In particular, there exists $\eps>0$ such that
\begin{align}\label{from contr}
\esssup_{(\partial U)^{\eps}\cap U}\rom{H}\big (\cdot,\mathrm D^{\vec{k}}u \big)+\delta\leq M.
\end{align}
By possibly recreasing $\eps$ if necessary, we can choose a function $\varphi_\eps\in C^\infty(\R^n)$, with $0\leq\varphi_\eps\leq1$, such that $\varphi_\eps=0$ in $U\setminus (\partial U)^\eps$ and $\varphi_\eps=1$ in $(\partial U)^{\eps/2}$. Now, for $\lambda\in(0,1)$, set
$$
u_\lambda:=\lambda u + (1-\lambda)\varphi_\eps u.
$$
note that $u_\lambda\in W^{k,\infty}{(U,\R^N)}$. Moreover, $u_\lambda=u$ on $(\partial U)^{\eps/2}\cap U$, so that 
\begin{align}\label{u lambda}
u_\lambda\in u+W^{k,\infty}_c(U,\R^N)\subseteq W_u^{k,\infty}(U,\R^N).
\end{align}
Let us compute $\mathrm D^{\vec{k}} u_\lambda$.  For every $h=1,\ldots, k$, we have
\begin{align}\label{deriv product}
\mathrm D^hu_\lambda=\lambda \mathrm D^h u+(1-\lambda)\mathrm D^h(\varphi_\eps u).
\end{align}
By the Leibniz formula,
\begin{align*}
\partial ^h_{i_1,\ldots,i_h}(\varphi_\eps u)=\sum_{j=0}^h \binom{h}{j}(\partial^{h-j}_{i_{j+1},\ldots, i_h} \varphi_\eps)(\partial^{j}_{i_{1},\ldots, i_j} u).
\end{align*}
In particular, for every $h=0,\ldots, k$, we can estimate
$$
|\mathrm D^h(\varphi_\eps u)|\leq 2^h\|\varphi_\eps\|_{W^{h,\infty}(U)}\|u\|_{W^{h,\infty}(U,\R^N)}\qquad \mbox{in }U, 
$$
which gives
\begin{align}\label{crude est}
\big| \mathrm D^{\vec{k}}(\varphi_\eps u) \big|\leq 2^{k+1}\|\varphi_\eps\|_{W^{k,\infty}(U)}\|u\|_{W^{k,\infty}(U,\R^N)}\qquad \mbox{in }U.
\end{align}
Putting together \eqref{deriv product} and \eqref{crude est}, we obtain
\begin{align}\label{fund est}
\big|  \mathrm D^{\vec{k}}u_\lambda-\lambda \mathrm D^{\vec{k}}u \big|  \leq(1-\lambda)2^{k+1}\|\varphi_\eps\|_{W^{k,\infty}(U)}\|u\|_{W^{k,\infty}(U,\R^N)}\qquad\mbox{in }U.
\end{align}
In particular, if we set 
$$C_1=C_1(\eps,k,\lambda)=(1-\lambda)2^{k+1}\|\varphi_\eps\|_{W^{k,\infty}(U)}+\lambda,$$
then
$$
\big|  \mathrm D^{\vec{k}}u_\lambda \big|  \leq C_1\|u\|_{W^{k,\infty}(U,\R^N)}\qquad\mbox{in }U.
$$
Now, we recall assumption \eqref{uniformity} and apply it to 
$$
R:=C_1\|u\|_{W^{k,\infty}(U,\R^N)}, \qquad\mathrm D^{\vec{k}}u_\lambda(x),\,\lambda\mathrm D^{\vec{k}}u(x)\in \overline{\mathbb B}_R(0)\quad\mbox{for a.e.   }x\in U,
$$
obtaining the existence of a continuous function $\omega=\omega_R:(0,\infty)\longrightarrow (0,\infty)$, non-decreasing, with $\omega(0^+)=0$, such that
\begin{align*}
\Big| \rom{H}(x,\mathrm D^{\vec{k}}u_\lambda(x))-\rom{H}(x,\lambda\mathrm D^{\vec{k}}u(x))\Big| \leq \omega\big(\big|\mathrm D^{\vec{k}}u_\lambda(x)-\lambda\mathrm D^{\vec{k}}u(x) \big|\big)\qquad\mbox{ a.e. }x\in U.
\end{align*} 
From \eqref{fund est} and the monotonicity of $\omega$, setting
$$
C_2=C_2(\eps,k)=2^{k+1}\|\varphi_\eps\|_{W^{k,\infty}(U)}\|u\|_{W^{k,\infty}(U,\R^N)},
$$
we get
\begin{align*}
\Big| \rom{H}(\cdot,\mathrm D^{\vec{k}}u_\lambda)-\rom{H}(\cdot,\lambda\mathrm D^{\vec{k}}u) \Big|\leq \omega\left((1-\lambda)C_2\right)\qquad\mbox{ a.e. in } U.
\end{align*}
By restricting this estimate to $(\partial U)^\eps\cap U$ and passing to the essential supremum on this set, we have
\begin{align*}
\esssup_{(\partial U)^\eps\cap U}\rom{H}(\cdot, \mathrm D^{\vec{k}}u_\lambda)&\leq\esssup_{(\partial U)^\eps\cap U}\rom{H}(\cdot, \lambda\mathrm D^{\vec{k}}u)+ \omega\left((1-\lambda)C_2\right)\\
&\leq M-\delta +\omega\left((1-\lambda)C_2\right),
\end{align*}
where in the last inequality we have applied \eqref{from contr}. Now, for $\lambda$ sufficiently close to $1$, by continuity of $\omega$, we have $\omega\left((1-\lambda)C_2\right)\leq\delta/2$, giving
\begin{align}\label{first key}
\esssup_{(\partial U)^\eps\cap U}\rom{H}(\cdot, \mathrm D^{\vec{k}}u_\lambda)\leq M-\frac{\delta}{2}<M.
\end{align}
On the other hand, let us prove that
\begin{align}\label{second key}
\esssup_{U\setminus (\partial U)^\eps}\rom{H}(\cdot, \mathrm D^{\vec{k}}u_\lambda)<M.
\end{align}
To this end note that, since $\varphi_\eps=0$ in $U\setminus (\partial U)^\eps$, we have $u_\lambda=\lambda u$ in $U\setminus (\partial U)^\eps$. Thus, we set 
$$
M_\lambda:=\esssup_U \rom{H}(\cdot, \lambda \mathrm D^{\vec{k}}u)
$$ 
and aim at proving $M_\lambda<M$. Without loss of generality, we can assume $M_\lambda>0$. So, by the definition of the essential supremum, for every $\delta\in (0,M_\lambda/2)$, there exists $x_\delta\in U$ such that
\begin{align}\label{def ess}
 \rom{H}\big(x_\delta,\lambda \mathrm{D}^{\vec{k}}u(x_\delta)\big)\geq M_\lambda-\delta.
\end{align}
On the other hand, by recalling assumption \eqref{strong mon} and \eqref{M}, we have
\begin{align}
 \rom{H}\big(x_\delta,\lambda \mathrm{D}^{\vec{k}}u(x_\delta)\big)
 &\leq \rom{H}(x_\delta, \mathrm{D}^{\vec{k}}u(x_\delta))-c(\lambda, |\mathrm D^{\vec{k}}u(x_\delta)|)\nonumber\\
&\leq M-c(\lambda, |\mathrm D^{\vec{k}}u(x_\delta)|).\label{from strong}
\end{align}
Now we claim that there exists $\sigma>0$ independent of $\delta$ such that
\begin{align}\label{c}
c(\lambda,|\mathrm D^{\vec{k}}u(x_\delta)|)\geq\sigma.
\end{align}
Indeed, from \eqref{def ess} and the continuity of $\rom{H}$ at $X=0$, we infer that $|\mathrm D^{\vec{k}}u(x_\delta)|$ is bounded away from $0$. Moreover, since $u\in W^{k,\infty}(\Om,\R^N)$, we have that $|\mathrm D^{\vec{k}}u(x_\delta)|$ is also bounded from above. In other words, there exist two constants $R\geq r>0$, independent of $\delta$, such that
$$
r\leq| \mathrm D^{\vec{k}}u(x_\delta)|\leq R.
$$
Hence, since $c$ is continuous and strictly positive in $(0,1)\times(0,\infty)$, we have $$c(\lambda,| \mathrm D^{\vec{k}}u(x_\delta)|)\geq\min_{\rho\in[r,R]}c(\lambda,\rho)=:\sigma>0.$$
Putting together \eqref{def ess}, \eqref{from strong}, and \eqref{c}, we obtain
$$
M_\lambda-\delta\leq M-c(\lambda, |\mathrm D^{\vec{k}}u(x_\delta)|)\leq M-\sigma,
$$
for $\sigma>0$ independent of $\delta$. We conclude that $M_\lambda\leq M-\sigma<M$, by letting $\delta\to  0$.\\
As a consequence, we deduce \eqref{second key}, indeed
\begin{align*}
\esssup_{U\setminus (\partial U)^\eps}\rom{H}(\cdot, \mathrm D^{\vec{k}}u_\lambda)&=\esssup_{U\setminus (\partial U)^\eps}\rom{H}(\cdot, \lambda\mathrm D^{\vec{k}}u)\\
&\leq \esssup_{U}\rom{H}(\cdot, \lambda\mathrm D^{\vec{k}}u)\\
&=M_\lambda<M
\end{align*}
Finally, by recalling \eqref{M}, and the inequalities \eqref{first key} and \eqref{second key}, we obtain
$$
\esssup_{U}\rom{H}(\cdot, \mathrm D^{\vec{k}}u_\lambda)\leq \max\left\{\esssup_{(\partial U)^\eps\cap U}\rom{H}(\cdot, \mathrm D^{\vec{k}}u_\lambda), \esssup_{U\setminus (\partial U)^\eps}\rom{H}(\cdot, \mathrm D^{\vec{k}}u_\lambda)\right\}<M.
$$
This inequality and \eqref{u lambda} contradict the absolute minimality of $u$.\\
The proof is complete. \qed
\medskip

A straightforward consequence is Theorem \ref{thm2}, whose proof is detailed below.

\medskip

\noindent \textit{Proof of Theorem \ref{thm2}:} Let us show that satisfying the energy maximum principle implies the membership to the class $ \mathrm {AM}_c(\Om,\R^N)$. Fix $u\in W^{k,\infty}(\Om,\R^N)$, an open set $U\subseteq\Om$ and a map $\phi\in W^{k,\infty}_c(U,\R^N)$. Let $U'\Subset U$ be an open subset with supp$(\phi)\subseteq \overline{U'}$. Then we have
\begin{equation}\label{1st ine}
\begin{aligned}
\mathrm E_\infty(u+\phi,U)&=\esssup_U \rom{H}(\cdot, \mathrm D^{\vec{k}}(u+\phi))\\
&\geq \esssup_{U\setminus U'}\rom{H}(\cdot, \mathrm D^{\vec{k}}(u+\phi))\\
&=\esssup_{U\setminus U'}\rom{H}(\cdot, \mathrm D^{\vec{k}}u)
\end{aligned}
\end{equation}
note that $(\partial U)^\rho\cap U\subseteq U\setminus U'$ for a sufficiently small $\rho>0$. Hence
$$
\esssup_{U\setminus U'}\rom{H}(\cdot, \mathrm D^{\vec{k}}u)\geq \esssup_{(\partial U)^\rho\cap U}\rom{H}(\cdot, \mathrm D^{\vec{k}}u).
$$
By passing to the limit as $\rho\to   0$ in the right hand side and invoking Lemma \ref{lemma esssup}, we get
\begin{equation}\label{2nd ine}
\esssup_{U\setminus U'}\rom{H}(\cdot, \mathrm D^{\vec{k}}u)\geq \lim_{\rho\to  0}\esssup_{(\partial U)^\rho\cap U}\rom{H}(\cdot, \mathrm D^{\vec{k}}u)=\sup_{\partial U}\rom{H}(\cdot, \mathrm D^{\vec{k}}u)^*.
\end{equation}
Thus, if $u$ satisfies the energy maximum principle \eqref{MP}, by putting together \eqref{1st ine} and \eqref{2nd ine}, we get
\begin{align*}
\mathrm E_\infty(u+\phi,U)&\geq\sup_{\partial U}\rom{H}(\cdot, \mathrm D^{\vec{k}}u)^*\\
&=\esssup_{U}\rom{H}(\cdot, \mathrm D^{\vec{k}}u)\\
&=\mathrm E_\infty(u,U),
\end{align*}
yielding that $u\in \mathrm{AM}_c(\Om,\R^N)$.\\
For the reverse implication, one can just replicate the arguments in the proof of Theorem \ref{thm1}. Indeed, inequality \eqref{one ineq} holds for every $u\in W^{k,\infty}(\Om,\R^N)$. The opposite inequality is obtained exactly with the same reasoning, since the competitor $u_\lambda$ belongs to the class $u+W^{k,\infty}_c(U,\R^N)$ (recall \eqref{u lambda}). \qed
\medskip

As a consequence of Theorem \ref{thm2}, we obtain Corollary \ref{cor}.
\medskip

\noindent {\it Proof of Corollary} \ref{cor}: For $\Lambda>\mathrm E_\infty(u_0,\Om)$, consider the Dirichlet problem
\begin{equation}\nonumber
\left\{
\begin{aligned}
&\rom H(\cdot, \rom D^{\vec{k}}u)=\Lambda, &\text{ a.e. in }\Om,
\\
&u=u_0, & \text{ on }\partial \Omega.
\end{aligned}
\right.
\end{equation}
By \cite[Proposition 1]{K5} (which is consequence of results from \cite{D-M}), it admits infinitely many strong solutions $u\in W^{k,\infty}_{u_0}(\Om,\R^N)$. Each one of them trivially satisfies the energy maximum principle for $\mathrm E_\infty$ in $\Om$:
$$
\esssup_U \rom{H}\big (\cdot,\mathrm D^{\vec{k}}u \big)=\sup_{\partial U}\rom{H}(\cdot,\mathrm D ^{\vec{k}}u)^*=\Lambda\qquad\forall\  U\subseteq\Om \mbox{ open}.
$$
By Corollary \ref{thm2}, we have also that $u\in \mathrm{AM}_c(\Om,\R^N)$. \qed

\medskip

Now we turn our attention to Theorem \ref{thm p} concering the case of $p<\infty$. The content of the following lemma is probably well known, but we present its proof for sake of completeness.

\begin{Lemma}\label{smoothness}
Let $\Om\subseteq\R^n$ be an open set and $u:\Om\longrightarrow \R^N$ be a $p$-harmonic map, $p\in[2,\infty)$. Set $\mathscr{C}:=\{x\in\Om:\mathrm Du(x)=0\}$. Then $u\in C^\infty(\Omega\setminus\mathscr C, \R^N)$.
\end{Lemma}
\begin{proof}
Define $F:\R^{N\times n}\longrightarrow  \R^{N\times n}$ as $F(X):=|X|^{p-2}X$. Then, we have
$$
\rom DF(X)=(p-2)|X|^{p-4}X\otimes X+|X|^{p-2}I,
$$
where $I$ is the identity map in $(\R^{N\times n})^{\otimes 2}$.
Using the following identity
$$
(X\otimes X)Y:Z=(X:Y)(X:Z)\qquad\forall\  X,Y,Z\in \R^{N\times n},
$$ 
we have that
$$
\rom DF(X) Y:Z=Y:\rom DF(X)Z\qquad\forall\  X,Y,Z\in \R^{N\times n},
$$
i.e.\ $\rom DF(X)$ is a symmetric $4$-tensor.
Moreover, for every $X,Y\in \R^{N\times n}$, we have
$$
|X|^{p-2}|Y|^2\leq \rom DF(X)Y:Y=(p-2)|X|^{p-4}(X:Y)^2+|X|^{p-2}|Y|^2\leq(p-1)|X|^{p-2}|Y|^2.
$$
In particular, if we assume $0<\lambda\leq|X|\leq\Lambda$, then $DF(X)$ is elliptic in the Legendre sense (see \cite{GM}), since
\begin{align}\label{ellip}
\lambda^{p-2}|Y|^2\leq \rom DF(X)Y:Y\leq(p-1)\Lambda^{p-2}|Y|^2.
\end{align}
Now observe that $u$ solves (distributionally)
\begin{align}\label{eq F}
\mathrm{div}(F(\rom Du))=0 \qquad\mbox{in }\Om.
\end{align}
For every fixed $U'\Subset U\Subset\Omega\setminus\mathscr C$ and $i\in\N$, the difference quotient of a function $f:U\longrightarrow  \R$
$$
\rom D_h^if(x):=\frac{f(x+he_i)-f(x)}{h}
$$
is well defined on $U'$, for every $h\in \R$ with $0<|h|<{\rm dist}(U',\partial U)$. By restricting \eqref{eq F} to $U$, we can apply the difference quotient and obtain
$$
\mathrm{div}(\rom D_h^i(F(\rom Du)))=0 \qquad\mbox{in }U'.
$$
For $x\in U'$, we write
\begin{align*}
\rom D_h^i(F(\rom Du))&=\frac{1}{h}\int_0^1\frac{d}{dt} \left[F\left(\rom Du(x)+th\rom D_h^i\rom Du(x)\right)\right]dt\\
&=\left(\int_0^1\rom DF\left(\rom Du(x)+th\rom D_h^i\rom Du(x)\right)dt\right) \rom D_h^i \rom Du(x).
\end{align*}
Observe that
$$
\mathbb A(x):=\int_0^1\rom DF\left(\rom Du(x)+th\rom D_h^i\rom Du(x)\right)dt
$$
is a symmetric $4$-tensor field on $U'$. Let us prove that it is Legendre elliptic: first note that there exists $\lambda>0$ such that $|\rom Du|\geq\lambda$ in $U'$. In addition, since $u\in C^{1,\alpha}(U,\R^N)$ (see \cite{U}), for $|h|$ sufficiently small we have
$$
|\rom Du(x+he_i)-\rom Du(x)|\leq\frac{\lambda}{2}\qquad\forall\  x\in U'.
$$
Thus, for $t\in[0,1]$, we estimate
$$
\left|\rom Du(x)+th\rom D_h^i\rom Du(x)\right|=\left|\rom Du(x)+t\left(\rom Du(x+he_i)-\rom Du(x)\right)\right|\geq\lambda-t\frac{\lambda}2\geq\frac{\lambda}2\qquad\forall\  x\in U'.
$$
On the other hand, 
$$
\left|\rom Du(x)+th\rom D_h^i\rom Du(x)\right|\leq 3\|\rom Du\|_{L^\infty(U)}\qquad\forall\  x\in U'.
$$
From \eqref{ellip}, we deduce the Legendre ellipticity of $\mathbb A$ in $U'$ (with ellipticity constants independent of $h$).\\
Therefore, we have proved that $w:=\rom D_h^i \rom Du$ solves the linear elliptic system
$$
{\rm div}(\mathbb Aw)=0\qquad\mbox{in } U'.
$$
By Schauder estimates \cite[Thm. 5.19]{GM}, since $\mathbb A$ has $\alpha$-H\"older coefficients (recall $u\in C^{1,\alpha}(U,\R^N)$), we have $w\in C^{1,\alpha}(U', \R^{N\times n})$. The properties of the difference quotient ensure $\rom D u\in C^{2,\alpha}(U',\R^{N\times n})$. By a bootstrap argument and arbitrariness of $U'$, we can conclude that $u\in C^\infty(\Om\setminus\mathscr C, \R^N)$.
\end{proof}

Now, we are ready to prove Theorem \ref{thm p}.

\medskip

\noindent \textit{Proof of Theorem \ref{thm p}:} First, we observe that if $p=2$, the result is straightforward: every harmonic map $u:\Om\longrightarrow \R^N$ is smooth and satisfies $\Delta|\rom Du|^2\geq0$ on $\Om$.
So, we may assume $p>2$ in the following. \\
We divide the proof in two steps: firstly, we prove the gradient maximum principle outside the closed set $\mathscr{C}=\{x\in\Om:\mathrm Du(x)=0\}$; second, we show the argument for a generic open set $U\Subset\Om$.\\
Of course, we can assume that $\Om\setminus\mathscr C\neq \varnothing$, otherwise $u$ is constant on $\Om$.

\medskip

\noindent \textit{Step 1:} assume $U\Subset\Om\setminus\mathscr{C}$.\\
 In the following, the symbols $\cdot$ and $ \langle\cdot,\cdot\rangle$ stand for the inner products in $\R^n$ and $\R^N$ respectively, while : stands for the Hilbert-Schmidt inner product in $\R^{N\times n}$. By Lemma \ref{smoothness}, we have $u\in C^\infty(\overline U,\R^N)$\footnote{By this notation, we mean that $u\in C^\infty(U',\R^N)$ for some $U\Subset U'\subseteq \Om\setminus\mathscr C$.}. So, we can differentiate \eqref{p eq} with respect to $x_i$, obtaining
\begin{align*}
0&=\partial_i\left[\rom{div}\big(|\rom Du|^{{p-2}}\rom Du\big)\right]=\rom{div}\left((p-2)|\rom Du|^{{p-4}}(\rom Du:\rom D\partial_i u)\rom Du+|\rom Du|^{{p-2}}\rom D\partial_iu\right)\qquad\mbox{in }U.
\end{align*}
By taking the inner product with $\partial_i u$, we have
\begin{equation}\label{inner}
\begin{aligned}
0&=\left\langle\rom{div}\left((p-2)|\rom Du|^{p-4}(\rom Du:\rom D\partial_i u)\rom Du+|\rom Du|^{{p-2}}\rom D\partial_i u\right),\partial_i u\right\rangle\\
&=\rom{div}\left((p-2)|\rom Du|^{p-4}(\rom Du:\rom D\partial_i u)\langle\rom Du,\partial_iu\rangle+|\rom Du|^{p-2}\langle\rom D\partial _iu,\partial_iu\rangle\right)-\\
&\quad-(p-2)|\rom Du|^{p-4}(\rom Du:\rom D\partial_i u)^2-|\rom Du|^{p-2}|\rom D\partial_iu|^2, \qquad\mbox{in }U.
\end{aligned}
\end{equation}
From \eqref{inner}, we deduce
\begin{equation}\label{inner2}
\begin{aligned}
\rom{div}\left((p-2)|\rom Du|^{p-4}(\rom Du:\rom D\partial_i u)\langle\rom Du,\partial_iu\rangle+|\rom Du|^{p-2}\langle\rom D\partial _iu,\partial_iu\rangle\right)\geq0, \qquad\mbox{in }U.
\end{aligned}
\end{equation}
By taking the sum over $i=1,\ldots,n$ in \eqref{inner2} and observing that
$$
\sum_{i=1}^n(\rom Du:\rom D\partial_i u)\langle\rom Du,\partial_iu\rangle=\left(\rom Du^\rom T\rom Du\right)\cdot\langle\rom D^2u, \rom Du\rangle,
$$
we obtain
\begin{equation}\label{inner3}
\begin{aligned}
\rom{div}\left((p-2)|\rom Du|^{p-4}\left(\rom Du^\rom T\rom Du\right)\cdot\langle\rom D^2u, \rom Du\rangle+|\rom Du|^{p-2}\langle\rom D^2u,\rom Du\rangle\right)\geq0, \qquad\mbox{in }U.
\end{aligned}
\end{equation}
Now, we define 
$$
f:=\frac1p{|\rom Du|^p}, \qquad \rom A:= I+(p-2)\frac{\rom Du^\rom T\rom Du}{|\rom Du|^2},
$$
where $I$ is the identity matrix of $\R^{n\times n}$. Thus, 
$$
\rom Df=|\rom Du|^{p-2}\langle\rom D^2u, \rom Du\rangle,
$$
so that we can rewrite \eqref{inner3} as
\begin{equation}\label{inner4}
\begin{aligned}
\rom{div}(\rom A\rom Df)\geq0, \qquad\mbox{in }U.
\end{aligned}
\end{equation}
Since $u\in C^\infty(\overline U,\R^N)$, the function $f$ is a classical subsolution to a linear elliptic partial differential inequality in divergence form, with smooth and bounded coefficients. By the maximum principle for elliptic PDEs, we have
$$
\max_{\overline U}f=\max_{\partial U}f,
$$
yielding \eqref{GMP}.

\medskip

\noindent
\textit{Step 2}: general case.\\
Let $U\subseteq\Om$ be an open subset. Then, without loss of generality, we can assume that $U\not\subseteq\mathscr C$. Indeed, if $U\subseteq\mathscr C$, then also $\overline U\subseteq\mathscr C$, so that \eqref{GMP} trivially holds. Therefore, since $U\setminus \mathscr C$ is a non-empty open subset of $\Om$, there exists $\eps>0$ such that $U_\eps:=U\setminus\overline{\mathscr C^\eps}$ is a non-empty open subset of $\Om$, where ${\mathscr C^\eps}$ is the open $\eps$-neighbourhood of ${\mathscr C}$. Thanks to \textit{Step 1}, since $U_\eps\Subset\Om\setminus \mathscr C$, we have
\begin{align}\label{GMP eps}
\max_{\overline U_\eps}|\rom Du|=\max_{\partial U_\eps}|\rom Du|.
\end{align}
Now, set $M:=\max_{\overline U}|\rom Du|>0$. By possibly reducing $\eps$ in size, the continuity of $\rom Du$ ensures
\begin{align}\label{eps small}
\max_{\overline{\mathscr C^\eps}}|\rom Du|< M.
\end{align}
Hence
$$
M=\sup_{U}|\rom Du|=\max\left\{\sup_{U_\eps}|\rom Du|,\max_{\overline{\mathscr C^\eps}}|\rom Du| \right\}=\sup_{U_\eps}|\rom Du|= \max_{ \overline{U}_\eps}|\rom Du|.
$$
In particular, from \eqref{GMP eps} we have
$$
\max_{\partial U_\eps}|\rom Du|=M.
$$
On the other hand, we can write
$$
\partial U_\eps=(\partial U\setminus\overline{\mathscr C^\eps})\cup (U\cap \partial{\mathscr C^\eps}),
$$
so that, using \eqref{eps small}, 
$$
M=\sup_{\partial U_\eps}|\rom Du|=\max\left\{\sup_{\partial U\setminus\overline{\mathscr C^\eps}}|\rom Du|, \sup_{U\cap \partial{\mathscr C^\eps}}|\rom Du|\right\}\leq\max\left\{\max_{\partial U}|\rom Du|, \max_{ \overline{\mathscr C^\eps}}|\rom Du|\right\}=\max_{\partial U}|\rom Du|,
$$
which ensures the validity of \eqref{GMP} .\qed\\

\noindent \textsc{Acknowledgements:} S.C and N.K. acknowledge partial financial support through the EPSRC grant EP/X017109/1. S.C and R.M. acknowledge partial financial support through the EPSRC grant EP/X017206/1. S.C. is a member of Gruppo Nazionale per
l’Analisi Matematica, la Probabilità e le loro Applicazioni (GNAMPA) of the Istituto Nazionale
di Alta Matematica (INdAM) of Italy.\\

\noindent \textsc{Data availability}: Our manuscript has no associated data.\\

\noindent \textsc{Conflict of interest}: The authors have no conflict of interest to declare for this article.


\vspace{3mm}
\textsc{Simone Carano}\\
Department of Mathematics and Statistics, University of Reading\\
Whiteknights Campus, Pepper Lane, Reading RG6 6AX, UK\\
E-mail: s.carano@reading.ac.uk\\

\textsc{Nikos Katzourakis}\\
Department of Mathematics and Statistics, University of Reading\\
Whiteknights Campus, Pepper Lane, Reading RG6 6AX, UK\\
E-mail: n.katzourakis@reading.ac.uk\\

\textsc{Roger Moser}\\
Department of Mathematical Sciences, University of Bath\\
Bath BA2 7AY, UK\\
E-mail: r.moser@bath.ac.uk

\end{document}